\documentclass [reqno,10pt,oneside]{amsart}

\usepackage{amsthm}
\usepackage{amscd}
\usepackage{amsmath,amssymb}
\usepackage{algorithm}
\usepackage[noend]{algorithmic}
\usepackage{graphicx}
\usepackage{breqn}

\theoremstyle {plain}
\newtheorem {thm}{Theorem}[section]
\newtheorem {prop}[thm]{Proposition}

\theoremstyle {definition}

\theoremstyle {remark}

\newcommand{\gen}[1]{\langle #1 \rangle}

\newcommand{\singular}{{\sc Singular }}

\begin{document}

\bibliographystyle{alpha}

\title{Unimodular ICIS, A CLASSIFIER}

\author{Deeba Afzal}
\address{Deeba Afzal\\ Department of Mathematics\\ University of Kaiserslautern\\
Erwin-Schr\"odinger-Str.\\ 67663 Kaiserslautern\\ Germany.\newline Department of Mathematics\\ University of Lahore\\
Near Raiwind Road\\ Lahore 54000\\ Pakistan}
\email{deebafzal@gmail.com}

\author{Farkhanda Afzal}
\address{Farkhanda Afzal\\ School of Mathematics and systems science\\ Beihang University\\
SMSS, 37 Xueyuan Lu, Haitian District\\ Beijing 100096\\ China}
\email{farkhanda_imran@live.com}

\author{Sidra Mubarak}
\address{Sidra Mubarak\\ Department of Mathematics\\ University of Lahore\\
Near Raiwind Road\\ Lahore 54000\\ Pakistan}
\email{sidramubarak092@gmail.com }

\author{Gerhard Pfister}
\address{Gerhard Pfister\\ Department of Mathematics\\ University of Kaiserslautern\\
Erwin-Schr\"odinger-Str.\\ 67663 Kaiserslautern\\ Germany}
\email{pfister@mathematik.uni-kl.de}

\author{Asad Yaqub}
\address{Asad Yaqub\\ Department of Mathematics\\ University of Lahore\\
Near Raiwind Road\\ Lahore 54000\\ Pakistan}
\email{asad.yaqub11@gmail.com}

\thanks{$\textit{Mathematics subject classification}$: Primary 14B05; Secondary 14H20, 14J17.}
\thanks{$\textit{Reference number}$ \#3569}
\date{\today}

\maketitle

\begin{abstract}
We present the algorithms for computing the normal form of unimodular complete intersection surface singularities classified by C. T. C. Wall. He indicated in the list only the $\mu$-constant strata and not the complete classification in each case.
We give a complete list of surface unimodular singularities. We also give the description of a classifier which is implemented in computer algebra system \singular.
\end{abstract}

\vskip.5cm

\section{Introduction}
In this article we report about a classifier for unimodular isolated complete intersection surface singularities in the computer algebra system \singular (cf. {\cite{DGPS13}},{\cite{{GP07}}}).
 Marc Giusti gave the complete list of simple isolated complete intersection singularities which are not hypersurfaces (cf.\ \cite{GM83}). Wall achieved the classification of unimodular singularities which are not hypersurfaces (cf.\ \cite{Wal83}). Two of the authors described Giusti's classification in terms of certain invariants. Based on this description it is not necessary to compute the normal form for finding the type of the singularity. This is usually more complicated and may be space and time consuming (cf.\ \cite{ADPG1}, \cite{ADPG2}). Similarly the type of singularities in terms of certain invariants (Milnor number, Tjurina number and semi groups) for unimodular isolated complete intersection space curve singularities given by C. T. C. Wall is characterized \cite {ADPG3}.

 A basis for a classifier is a complete list of these singularities together with a list of invariants characterizing them. Since Wall gave only representatives of the $\mu$-constant strata in his classification (cf.\ \cite{Wal83}), we complete his list by computing the versal $\mu$-constant deformation of the singularities. The new list obtained in this way contains all unimodular complete intersection surface singularities.

 In section \ref{twojet} we characterize the $2$-jet of the unimodular complete intersection singularities using primary decomposition, Krull dimension and Hilbert polynomials.
 In section \ref{ctcwall} we give the complete list of unimodular complete intersection surface singularities by fixing the $2$-jet of the singularities and develop algorithms for each case. In section \ref{main} we give the main algorithm for computing the type of unimodular complete intersection surface singularity.


\section{Characterization of normal form of 2-jet of singularities} \label{twojet}
 Consider $I=\gen{f_{1},f_{2}} \subseteq \gen {x,y,z,w}^2\mathbb{C}[[x,y,z,w]]$ defining a complete intersection singularity and let $I_{2}$ be the $2$-jet of $I$.
 According to C.T.C. Wall's classification the $2$-jet of $\gen{f_{1},f_{2}}$ is a homogenous ideal generated by 2 polynomials of degree 2. We want to give a description of the type of a singularity without producing the normal form. C.T.C. Wall's classification is based on the classification of the $2$-jet $I_{2}$ of $\gen {f,g}$. Let $\bigcap^s_{i=1}Q_{i}$ be the irredundant primary decomposition of $I_{2}$ in $\mathbb{C}[[x,y,z,w]]$. Let $t$ be the number of prime ideals appearing in primary decomposition of $I_{2}$ and $j_{i}$ be the number of conjugates corresponding to each prime ideal.
 Let $d_{i}=dim_{\mathbb {C}}\mathbb {C}[[x,y,z,w]]/Q_{i}, i=1,...,s$ and $h_{i}$ be the Hilbert polynomial of $\mathbb {C}[[x,y,z,w]]/Q_{i}$. According to C.T.C. Wall's classification we obtain unimodular singularities only in the following cases.
\begin{table}[H]
\label{2jet}
$$
\begin{array}{|c|c|c|}
  \hline
  Name& Characterization & Normal form  \\\hline
  T_{2,2,2,2}&s=1,  & \gen{x^{2}+y^{2}+z^{2},y^{2}+\lambda z^{2}+w^{2}} \lambda \neq 0,1  \\
               &d_{1}=2,  h_{1}=4t              &               \\
                        & t=1, j_{1}=1, j_{2}=1            &                    \\
  T_{p,2,2,2} &s=1& \gen{xy+z^{2}+w^{2},zw+y^{2}}\\
  p>2& d_{1}=2, h_{1}=4t &  \\
             &t=1, j_{1}=1 &\\
             &A_{1}\,\,after\,\,blowing\,\, up\footnotemark\\

  T_{p,q,2,2} &s=2,  & \gen{xy+z^{2}+w^{2},zw} \\
     p, q>2                 &d_{1}=d_{2}=2&            \\
              &h_{1}=h_{2}=1+2t&          \\
  T_{p,q,r,2} &s=3,  & \gen{xy+w^{2},zw} \\
        p, q, r>2    &   d_{1}=d_{2}=d_{3}=2   &   \\
                     & h_{1}=h_{2}=1+t, h_{3}=1+2t &  \\
                          & t=3, j_{1}=j_{2}=j_{3}=1 &\\
                          & Q_{3} \nsubseteq Q_{1}+Q_{2}\,\,and\,\, Q_{3}\,\,has\,\,a\,\,generator\,\,of\,\,order\,\,2  &\\
  T_{p,q,r,s} &s=4, d_{1}=d_{2}=d_{3}=d_{4}=2 & \gen{xy,zw} \\
  p, q, r, s >2& h_{1}=h_{2}=h_{3}=h_{4}=1+t &  \\\hline

  I&s=4,d_{1}=d_{2}=d_{3}=d_{4}=1  & \gen{xy-xz, yz-yx}  \\
           & h_{1}=h_{2}=h_{3}=h_{4}=1+t   &             \\\hline
  J^{'}&s=1 & \gen{xy+z^{2}, w^{2}+xz}  \\
           &d_{1}=2, h_{1}=4t   &                  \\
           & t=1, j_{1}=1           &                    \\
            &A_{2}\,\,after\,\,blowing\,\, up\\\hline
  K^{'}&s=1  & \gen{xy+z^{2}, w^{2}+x^{2}}\\
           & d_{1}=2, h_{1}=4t   &                  \\
           &  t=1, j_{1}=2           &                   \\\hline
  L&s=2,d_{1}=d_{2}=2  & \gen{xy+z^{2}, w^{2}+xz}  \\
           & h_{1}=1+t, h_{2}=1+3t   &                  \\\hline
  M &s=3,d_{1}=d_{2}=d_{3}=2  & \gen{wy+x^{2}-z^{2}, wx}  \\
           & h_{1}=h_{2}=1+t , h_{3}=1+2t  &                  \\
            & t=3, j_{1}=j_{2}=j_{3}=1           &                    \\
            &  Q_{3}\subseteq Q_{1}+ Q_{2}\,\,and\,\, Q_{3}\,\,has\,\,a\,\,generator\,\,of\,\,order\,\,2 &\\\hline

\end{array}
$$
\caption{Charecterization table of normal forms of $2-jet$ of $I$}
\end{table}
\footnotetext{Here after blowing up means the singularity type of the strict transform of $I_{2}$ in the blowing up of $\gen{x,y,z,w}$.}
\newpage
\section{Unimodular complete intersection surface singularities}
\label{ctcwall}
We set

          \[
    l_i(x,y)=
\begin{cases}
    xy^{q},& \text{if } i=2q\\
    y^{q+2},& \text{if } i=2q+1
\end{cases}
\]
for brevity.
\subsection{I singularities}
Assume the $2-jet$ of $<\gen{f,g}$ has normal form $\gen{xy-xz,yz-xy}$. In this case according to C.T.C.Wall's classification the unimodular surface singularities with their Milnor number say $\mu$ and Tjurina number $\tau$ are given in the table below.
\begin{table}[H]
\[
\begin{array}{|l|c|c|c|}
\hline Name & Normal\,\,\ form &\mu  &\tau  \\
\hline  I_{1,0} & \gen{x(y-z)+w^{3},y(z-x)+\lambda w^{3}}\,\,\,\, \lambda \neq 0,1 & 13 & 13 \\\hline
        I_{1,i} & \gen{x(y-z)+w^{3}, y(z-x)+w^{2}l_{i-1}(x,w)}&13+i &13+i-2 \\
\hline
\end{array}
\]
\caption{}
\label{Iold}
\end{table}
\begin{prop}
\label{I1}
The unimodular complete intersection surface singularity with Milnor number $\mu=13$   are $I_{1,0}$ with Tjurina number $\tau=13$ defined by the ideal
\[\gen {xy-xz+w^{3},yz-xy+\lambda w^{3}}\]
and $I_{1,0,1}$ with Tjurina number $\tau=12$ defined by the ideal
\[\gen {xy-xz+w^{3},yz-xy+\lambda w^{3}+w^{4}}.\]
\end{prop}
\begin{proof}
In the list of C.T.C. Wall $I_{1,0}$ is the singularity defined by the ideal
$$\gen {xy-xz+w^{3},yz-xy+\lambda w^{3}}$$
with Milnor number $\mu=13$ and Tjurina number $ \tau=13$. The versal deformation of $I_{1,0}$ is given by
\begin{dmath*}{}
\gen{xy-xz+w^{3}+t_{1}zw+t_{2}w+t_{3}z+t_{4},yz-xy+\lambda w^{3}+\lambda_{1}w^{4}+\lambda_{2}w^{3}+
\lambda_{3}w^{2}+\lambda_{4}zw+\lambda_{5}yw+\lambda_{6}w
+\lambda_{7}z+\lambda_{8}y+\lambda_{9}}.
\end{dmath*}
$I_{1,0}$ defines a weighted homogenous singularity with weights, $(w_{1},w_{2},w_{3},w_{4}) = (6,6,6,4)$ and the degrees $(d_{1},d_{2})=(12,12)$. The versal $\mu$-constant deformation of $I_{1,0}$ is given by $$\gen{xy-xz+w^{3},yz-xy+\lambda w^{3}+\lambda_{1}w^{4}}.$$ Using the coordinate change $x\rightarrow \xi^{6}x$, $y\rightarrow \xi^{6}y$, $z\rightarrow \xi^{6}z$, $w\rightarrow \xi^{4}w$. We obtain
$$\gen{xy-xz+w^{3},yz-xy+\lambda w^{3}+\lambda_{1}\xi^{4}w^{4}}.$$
Choosing $\xi$ such that $\lambda_{1}\xi^{4}= 1$. So we obtain $$\gen{xy-xz+w^{3},yz-xy+\lambda w^{3}+w^{4}}$$ with Tjurina number $\tau = 12$ different from $I_{1,0}$.
\end{proof}
Summarizing the results of the above preposition we complete the list of unimodular complete intersection singularities in case of $\gen{f,g}$ having 2-jet with normal forms $\gen{xy-xz,yz-xy}$.
\begin{table}[H]
\[
\begin{array}{|l|c|c|c|}
\hline Type &Normal form & \mu  &\tau\\
\hline I_{1,0,1}      &    \gen{xy-xz+w^{3},yz-xy+\lambda w^{3}+w^{4}}  & 13   &  12 \\
\hline
\end{array}
\]
\caption{}
\label{Inew}
\end{table}
\begin{prop}
Let $(V(\langle f,g\rangle ),0)\subseteq (\mathbb{C}^{4},0)$ be the germ of a complete intersection surface singularity. Assume it is not a hypersurface singularity and the $2$-jet of $\langle f,g\rangle$ has normal form $\langle xy-xz,yz-xy\rangle$. $(V(\langle f,g\rangle ),0)$ is unimodular if and only if it is isomorphic to a complete intersection in Tables \ref{Iold} and \ref{Inew}.
\end{prop}
\begin{proof}
The proof is a direct consequence of C.T.C. Wall's classification and Propositions \ref{I1}.
\end{proof}
We summarize our approach in this case in Algorithm \ref{Ising}
\begin{algorithm}
\caption{\texttt{Isingularity(I)}}
\label{Ising}
\begin{algorithmic}[1]
\REQUIRE $I=\gen{f,g}\subseteq \gen{x,y,z,w}^2\mathbb{C}[[x,y,z,w]]$ such that $2$-jet  of $I$\\             has normal form $\gen{xy-xz,yz-xy}$
\ENSURE the type of the singularity
\vspace{0.1cm}
\STATE compute $\mu=$Milnor number of $I$;
\STATE compute $\tau=$Tjurina number of $I$;
\STATE compute $B=$the singularity type of the strict transform of $I$ in the blowing up of $\gen{x,y,z,w}$
\footnotemark
\IF{$\mu=13$ and $B=A[1]$}
\IF{$\mu-\tau=0$}
\RETURN $(I_{1,0})$;
\ENDIF
\IF{$\mu-\tau=1$}
\RETURN $(I_{1,0,1})$;
\ENDIF
\ENDIF
\IF{$\mu=13+i$, $i>0$ and $B=A[i+1]$}
\IF{$\mu-\tau=2$}
\RETURN $(I_{1,i})$;
\ENDIF
\ENDIF
\RETURN $(not\quad unimodular)$;
\end{algorithmic}
\end{algorithm}
\footnotetext{List $B$ gives the singularity type of the strict transform of $I$ in the blowing up of $\gen{x,y,z,w}$ and it returns 1 if it is smooth.}

\begin{table}[H]
\[
\begin{array}{|l|c|c|c|}
\hline Name & Normal\,\,\ form &\mu  &\tau  \\
\hline  T_{2,2,2,2} & \gen{x^{2}+y^{2}+z^{2},y^{2}+z^{2}+w^{2}} & 7 & 7 \\\hline
        T_{p,q,r,s} & \gen{xy+z^{2}+w^{2}, zw+x^{2}+y^{2}}&p+q+r+s-1 &p+q+r+s-2  \\
        p>2, q,r,s \geq 2           &           &                  &             \\\hline
\end{array}
\]
\caption{}
\label{Told}
\end{table}
\subsection{T singularities}
Assume the $2-jet$ of $\gen{f,g}$ has normal form $\gen{x^{2}+y^{2}+z^{2},y^{2}+z^{2}+w^{2}}$, $\gen{xy+z^{2}+w^{2},zw+y^{2}}$, $\gen{xy+z^{2}+w^{2},zw}$ or $\gen{xy,zw}$. In this case according to C.T.C.Wall's classification the unimodular surface singularities with Milnor number $\mu$ and Tjurina number $\tau$ are given in the table 4 above.

\begin{algorithm}[H]
\caption{\texttt{Tsingularity(I)}}
\label{Tsing}
\begin{algorithmic}[1]
\REQUIRE $I=\gen{f,g}\subseteq \gen{x,y,z,w}^2\mathbb{C}[[x,y,z,w]]$ such that the$ 2$-jet  of $I$           has normal form $\gen{xy,zw}$,$\gen{xy+z^{2}+w^{2},zw+x^{2}+y^{2}}$, $\gen{xy+z^{2}+w^{2},zw+y^{2}}$, $\gen{xy+z^{2}+w^{2},zw}$
\ENSURE the type of the singularity
\vspace{0.1cm}
\STATE compute $\mu=$Milnor number of $I$; $\tau=$Tjurina number of $I$ and $B=$the singularity type of the strict transform of $I$ in the blowing up of $\gen{x,y,z,w}$
\IF{$\mu=\tau$ and $2$-jet has normal form $\gen{xy+z^{2}+w^{2},zw+x^{2}+y^{2}}$  and B=1}
\RETURN $(T_{2,2,2,2})$;
\ENDIF
\IF{$\mu-\tau=1$}
\IF{$2$-jet has normal form $<\gen{xy+z^{2}+w^{2},zw+y^{2}}$}
\IF{$\mu=8$ and $B=1$}
\RETURN $(T_{3,2,2,2})$;
\ENDIF
\IF{$\mu>8$ and $B=A[\mu-8]$}
\RETURN $(T_{\mu-5,2,2,2})$;
\ENDIF
\ENDIF
\IF{$2$-jet has normal form $\gen{xy+z^{2}+w^{2},zw}$}
\IF{$B=1$}
\RETURN $(T_{3,2,3,2})$;
\ENDIF
\IF{$B=A[r-3]$ and $r>3$ }
\RETURN $(T_{3,2,r,2})$;
\ENDIF
\IF{$B=A[p-3],A[r-3]$ and $p,r>3$ }
\RETURN $(T_{p,2,r,2})$;
\ENDIF
\ENDIF
\IF{$2$-jet has normal form $<\gen{xy+w^{2},zw}$}
\IF{$B=1$}
\RETURN $(T_{3,3,3,2})$;
\ENDIF
\IF{$B=A[r-3]$ and $r>3$ }
\RETURN $(T_{3,3,r,2})$;
\ENDIF
\IF{$B=A[q-3],A[r-3]$ and $q,r>3$ }
\RETURN $(T_{3,q,r,2})$;
\ENDIF
\IF{$B=A[p-3],A[q-3],A[r-3]$ and $p,q,r>3$ }
\RETURN $(T_{p,q,r,2})$;
\ENDIF
\ENDIF
\IF{$2$-jet has normal form $\gen{xy,zw}$ }
\IF{$B=1$}
\RETURN $(T_{3,3,3,3})$;
\ENDIF
\IF{$B=A[s-3]$ and $s>3$ }
\RETURN $(T_{3,3,3,s})$;
\ENDIF
\IF{$B=A[r-3],A[s-3]$ and $r,s>3$ }
\RETURN $(T_{3,3,r,s})$;
\ENDIF
\IF{$B=A[q-3]A[r-3],A[s-3]$ and $q,r,s>3$ }
\RETURN $(T_{3,q,r,s})$;
\ENDIF
\IF{$B=A[p-3],A[q-3],A[r-3],A[s-3]$ and $p,q,r,s>3$ }
\RETURN $(T_{p,q,r,s})$;
\ENDIF
\ENDIF
\ENDIF
\RETURN $(not\quad unimodular)$;
\end{algorithmic}
\end{algorithm}
\subsection{$\textbf{J}'$ singularities}
Assume the $2-jet$ of $\gen{f,g}$ has normal form $\gen{xy+z^{2},w^{2}+xz}$. According to C.T.C. Wall's classification the unimodular surface singularities are given in the following table

\begin{table}[H]
$$
    \begin{array}{|c|c|c|c|}
      \hline
      Name & Normal\,\ Form&\mu &\tau \\\hline
     J'_{6m+9} & \gen{xy+z^{2},xz+w^{2}+y^{3m+3}} & 6m+9 &6m+9  \\\hline
      J'_{6m+10} & \gen{xy+z^{2},w^{2}+xz+zy^{2m+2}} & 6m+10 & 6m+10  \\\hline
     J'_{6m+11}  & \gen{xy+z^{2},w^{2}+xz+y^{3m+4}}& 6m+11 & 6m+11  \\\hline
      J'_{m+1,0} &\gen{xy+z^{2},w^{2}+xz+z^{2}y^{m}+\lambda y^{3m+2}} \lambda\neq 0,-4/27&6m+7 &6m+7 \\\hline
  J'_{m+1,i} &\gen{xy+z^{2},w^{2}+xz+z^{2}y^{m}+y^{3m+2-i}}&6m+7+i &6m+5+i \\\hline
    \end{array}
$$
\caption{}
\label{Jold}
\end{table}
\begin{prop}
\label{J6m+9}
The unimodular complete intersection surface singularities having Milnor number of the form $\mu=6m+9$ where m is a positive integer are $J'_{6m+9}$ with Tjurina number $\tau=6m+9$ defined by the ideal
\[\gen{xy+z^{2},xz+w^{2}+y^{3m+3}}\]
and $J'_{6m+9,i+1}$ with Tjurina number $\tau=6m+8-i$ are defined by the ideal
\[\gen {xy+z^{2},xz+w^{2}+y^{3m+3}+zy^{3m+2-i}}\,\, for\,\, i=0,1,\ldots,m.\]
\end{prop}
\begin{proof}
In C.T.C. Wall's list the only unimodular complete intersection surface singularities are the singularities $J'_{6m+9}$ defined by the ideal $$\gen{xy+z^{2}, xz+w^{2}+y^{3m+3}}.$$
The versal deformation of $J'_{6m+9}$ is given by $$\gen{xy+z^{2}+Aw+Bz+C, w^{2}+xz+y^{3m+3}+\sum^{3m+2}_{i=0}\alpha_{i}y^{3m+2-i}z+\sum^{3m+2}_{i=0}\beta_{i}y^{3m+2-i}}.$$ $J'_{6m+9}$ defines a weighted homogenous isolated complete intersection singularity with weights $$(w_{1},w_{2},w_{3},w_{4})=(12m+10,6,6m+8,9m+9)$$ and degrees $$(d_{1},d_{2})=(12m+16,18m+18).$$ The versal $\mu-$constant deformation of $J'_{6m+9}$ is $$\gen{xy+z^{2}, w^{2}+xz+y^{3m+3}+\sum^{m}_{i=0}\alpha_{i}y^{3m+2-i}z}.$$
Consider $\phi\in Aut_{\mathbb{C}}(\mathbb{C}[[x,y,z,w]])$ defined by $\phi(x)=\epsilon^{12m+10}x$, $\phi(y)=\epsilon^{6}y$, $\phi(z)=\epsilon^{6m+8}z$ and $\phi(w)=\epsilon^{9m+9}w$. If $\alpha_{m}\neq 0$, then we can write $I$ as $$I=\gen{xy+z^{2},w^{2}+xz+y^{3m+3}+y^{2m+2}z(\sum^{m-1}_{i=0}\alpha_{i}y^{m-i}+\alpha_{m})}.$$
Let $\mu_{m-1}=\sum^{m-1}_{i=0}\alpha_{i}y^{m-i}+\alpha_{m}$. Then $I$ can be written $$I=\gen{xy+z^{2},w^{2}+xz+y^{3m+3}+y^{2m+2}z\mu_{m}}.$$
By applying the transformation $\phi$ we get $$\gen{xy+z^{2}, w^{2}+xz+y^{3m+3}+\epsilon^{2}y^{2m+2}z\mu^{'}_{m-1}}.$$ Choosing $\varepsilon$ such that $\varepsilon^{2}\mu^{'}_{m-1}=1$, we obtain
$$\gen{xy+z^{2}, xz+w^{2}+y^{3m+3}+y^{2m+2}z}.$$
Now we assume that $\alpha_{m}=0$. This implies that \[I=\gen{xy+z^{2},w^{2}+xz+y^{3m+3}+y^{2m+3}z(\sum^{m-2}_{i=0}\alpha_{i}y^{m-1-i}+\alpha_{m-1})}.\]
Let $\mu_{m-2}=\sum^{m-2}_{i=0}\alpha_{i}y^{m-1-i}+\alpha_{m-1}$ then  we can have
$$I=\gen{xy+z^{2} ,w^{2}+xz+y^{3m+3}+y^{2m+3}z\mu_{m-2}}.$$
Now applying the transformation $\phi$ we get $I$ as
$$\gen{xy+z^{2}, w^{2}+xz+y^{3m+3}+\epsilon^{6}y^{2m+3}z\mu^{'}_{m-2}}.$$
Choosing $\varepsilon$ such that $\varepsilon^{6}\mu^{'}_{m-2}=1$, so we obtain
$$\gen{xy+z^{2}, xz+w^{2}+y^{3m+3}+y^{2m+3}z}.$$
Now we assume $\alpha_{m-1}=0$. Iterating in the same way we get $m+1$ different singularities defined by $$\gen{xy+z^{2}, w^{2}+xz+y^{3m+3}+y^{3m+2-i}z}$$ for $i=0,1,\ldots,m$ and $m\geq1$.
\end{proof}
\begin{prop}
\label{J6m+10}
The unimodular complete intersection surface singularities having Milnor number of the form $\mu=6m+10$ where m is a positive integer are $J'_{6m+10}$ with Tjurina number $\tau=6m+10$ defined by the ideal
\[\gen{xy+z^{2},xz+w^{2}+zy^{2m+2}}\] and $J'_{6m+10,i+1}$
with Tjurina numbers $6m+10-i$ is defined by the ideal
\[\gen{xy+z^{2}, xz+w^{2}+zy^{2m+2}+y^{4m+4-i}}\,\,for\,\, i=0,1,\ldots,m.\]
\end{prop}
\begin{prop}
\label{J6m+11}
The unimodular complete intersection surface singularities having Milnor number of the form $\mu=6m+11$ where m is a positive integer are $J'_{6m+11}$ with Tjurina number $\tau=6m+11$ defined by the ideal
\[\gen{xy+z^{2},w^{2}+xz+y^{3m+4}}\] and $J'_{6m+11,i+1}$ with Tjurina number $6m+10-i$ is defined by the ideal \[\gen{xy+z^{2},w^{2}+xz+y^{3m+4}+y^{3m+3-i}}\,\,for\,\, i=0,1,\ldots,m.\]
\end{prop}
\begin{proof}
The proofs of Propositions \ref{J6m+10} and \ref{J6m+11} can be done similarly to the proof of Proposition \ref{J6m+9}.
\end{proof}
Summarizing the above results we have the following table
\begin{table}[H]
\[
\begin{array}{|c|c|c|c|}
  \hline
  Type & Normal Form  & \mu &\tau \\\hline
  J'_{6m+9,i+1}&\gen{xy+z^{2},xy+z^{2}+y^{3m+3}+zy^{3m+2-i}}&6m+9&6m+8-i  \\
  J'_{6m+10,i+1} &\gen{xy+z^{2},w^{2}+xz+zy^{2m+2}+y^{4m+4-i}} &6m+10 &6m+9-i \\
  J'_{6m+11,i+1} &\gen{xy+z^{2},w^{2}+xz+y^{3m+4}+y^{3m+3-i}z}&6m+11 &6m+10-i \\
               &for\,\,\,\ i=0,1,...,m.                &              &\\\hline
\end{array}
\]
\caption{}
\label{Jnew}
\end{table}
As a consequence of C.T.C. Wall's classification and Propositions \ref{J6m+9} - \ref{J6m+11} we obtain:
\begin{prop}
\label{Jcase}
Let $(V(\langle f,g\rangle ),0)\subseteq (\mathbb{C}^{4},0)$ be the germ of a complete intersection surface singularity. Assume it is not a hypersurface singularity and the $2$-jet of $\langle f,g\rangle$ has normal form $\langle xy+z2,w^{2}+xz\rangle$. $(V(\langle f,g\rangle ),0)$ is unimodular if and only if it is isomorphic to a complete intersection in Tables \ref{Jold} and \ref{Jnew}.
\end{prop}
We summarize our approach in this case in Algorithm \ref{Jsing}.
\begin{algorithm}[H]
\caption{$\texttt{J}^{'}$\texttt{singularity(I)}}
\label{Jsing}
\begin{algorithmic}[1]
\REQUIRE $I=\gen{f,g}\subseteq \gen{x,y,z,w}^2\mathbb{C}[[x,y,z,w]]$ such that $2$-jet of $I$\\
has normal form $\langle xy+z^{2},w^{2}+xz\rangle$
\ENSURE the type of the singularity
\vspace{0.1cm}
\STATE compute $\mu=$Milnor number of $I$, $\tau=$Tjurina number of $I$ and $B=$the singularity type of the strict transform of $I$ in the blowing up of $\gen{x,y,z,w}$
\IF{$\mu\equiv$ $3$ mod $6$ and $B=E[(\mu-15)+6]$}
\IF{$\mu=\tau$}
\RETURN $(J^{'}_{\mu})$;
\ELSE
\RETURN $(J^{'}_{\mu, \mu-\tau})$;
\ENDIF
\ENDIF
\IF{$\mu\equiv$ $4$ mod $6$ and $B=E[(\mu-16)+7]$}
\IF{$\mu=\tau$}
\RETURN $(J^{'}_{\mu})$;
\ELSE
\RETURN $(J^{'}_{\mu, \mu-\tau})$;
\ENDIF
\ENDIF
\IF{$\mu\equiv$ $5$ mod $6$ and $B=E[(\mu-17)+6]$}
\IF{$\mu=\tau$}
\RETURN $(J^{'}_{\mu})$;
\ELSE
\RETURN $(J^{'}_{\mu, \mu-\tau})$;
\ENDIF
\ENDIF
\IF{$\mu\equiv$ $1$ mod $6$ and $B=J[(\mu-13)/6,0]$ and $\mu=\tau$}
\RETURN $(J^{'}_{(\mu-7)/6+1,0})$;
\ENDIF
\IF{$\mu\neq \tau$ and $B=J[(\mu-13)/6,\mu-\tau-1]$}
\RETURN $(J^{'}_{(\mu-7)/6+1, \mu-\tau-1})$;
\ENDIF
\RETURN $not\quad unimodular$;
\end{algorithmic}
\end{algorithm}
\newpage
\subsection{$\textbf{K}'$ singularities}
Assume the $2-jet$ of $\gen{f,g}$ has normal form $\gen{xy+z^{2},w^{2}+x^{2}}$. In this case according to C.T.C.Wall's classification the unimodular surface singularities with Milnor number $\mu$ are given in the table below.
\begin{table}[H]
\[
\begin{array}{|l|c|c|c|}
\hline Name & Normal form &\mu &\tau \\
\hline  K^{'}_{10}   & xy+z^{2},w^{2}+x^{2}+y^{3} &10 &10\\\hline
       K^{'}_{11}   & xy+z^{2},x^{2}+w^{2}+zy^{2}&11& 11 \\\hline
       K^{'}_{1,0}  & xy+z^{2},x^{2}+w^{2}+z^{2}y+\lambda y^{4} \,\,\,\,(\lambda \neq 0,\frac{1}{4})&13&13 \\\hline
       K^{'}_{1,i}  & xy+z^{2},x^{2}+w^{2}+z^{2}y+ y^{4+i}&13+i&13+i-2\\\hline
       K^{b}_{1,i}  & xy+z^{2},x^{2}+w^{2}+2z^{2}y+ y^{4}+zyI_{i}(z,y)&13+i&13+i-2\\\hline
       K^{'}_{15}   & xy+z^{2},x^{2}+w^{2}+zy^{3}&15&15 \\\hline
       K^{'}_{16}   & xy+z^{2},x^{2}+w^{2}+y^{5}&16&16 \\
\hline
\end{array}
\]
\caption{}
\label{Kold}
\end{table}

\begin{prop}
\label{K10}
The unimodular complete intersection surface singularities with Milnor number $\mu=10$   are $K^{'}_{10}$ with Tjurina number $\tau=10$ defined by the ideal  \[\gen{xy+z^{2},w^{2}+x^{2}+y^{3}}\] and $K^{'}_{10,1}$ with Tjurina number $\tau=9$ defined by the ideal
\[\gen{xy+z^{2},w^{2}+x^{2}+y^{3}+yz^{2}}.\]
\end{prop}
\begin{proof}
In the list of C.T.C Wall $K^{'}_{10}$ is the singularity defined by the ideal \[\gen{xy+z^{2},w^{2}+x^{2}+y^{3} }\] with Milnor number $\mu=10$ and Tjurina number $ \tau=10$.
The versal deformation of $K^{'}_{10}$ is
\[\gen{ xy+z^{2}+t_{1}w+t_{2}y+t_{3},x^{2}+w^{2}+y^{3}+\lambda_{1}yz^{2}+\lambda_{2}z^{2}+
\lambda_{3}yz+\lambda_{4}z+\lambda_{5}y^{2}+\lambda_{6}y+\lambda_{7}}.\]
$K^{'}_{10}$ defines a weighted homogenous isolated complete intersection singularity with \[(w_{1},w_{2},w_{3},w_{4}) = (6,4,5,6)\,\, and \,\,the\,\, degrees\,\, (d_{1},d_{2})=(10,12).\]
The versal $\mu$-constant deformation of $K^{'}_{10,1}$ is given by $\langle xy+z^{2},w^{2}+x^{2}+y^{3}+\lambda_{1}yz^{2} \rangle$. Using the coordinate change
$x\rightarrow \xi^{6}x$, $y\rightarrow \xi^{4}y$, $z\rightarrow \xi^{5}z$ and $w\rightarrow \xi^{6}w$, we obtain
\[\gen{xy+z^{2},w^{2}+x^{2}+y^{3}+\lambda_{1}\xi^{2}yz^{2}}.\]
Choosing $\xi$ such that $\lambda_{1}\xi^{2}= 1$,
we obtain \[\gen{xy+z^{2},w^{2}+x^{2}+y^{3}+yz^{2}}\] with Tjurina number $\tau = 9$.
\end{proof}
\begin{prop}
\label{K11}
The unimodular complete intersection surface singularities with Milnor number $\mu=11$ are $K^{'}_{11}$ with Tjurina number $\tau=11$ defined by the ideal \[\gen{xy+z^{2},x^{2}+w^{2}+zy^{2}}\] and $K^{'}_{11,1}$ with Tjurina number $\tau=10$ defined by the ideal \[\gen{xy+z^{2},x^{2}+w^{2}+zy^{2}+y^{5}}.\]
\end{prop}
\begin{proof}
The proof of Proposition \ref{K11} can be done similarly to the proof of Proposition \ref{K10}.
\end{proof}
\begin{prop}
The unimodular complete intersection surface singularities with Milnor number $\mu=13$   are $K^{'}_{1,0}$ with Tjurina number $\tau=13$ defined by the ideal
\[\gen{xy+z^{2},x^{2}+w^{2}+z^{2}y+\lambda y^{4}},\,\,\,\,\,\,\,\,\lambda \neq 0,\frac{1}{4}.\]
 $K^{'}_{1,0,1}$ with Tjurina number $\tau=12$ defined by the ideal
 \[\gen{ xy+z^{2},x^{2}+w^{2}+z^{2}y+\lambda y^{4}+y^{5}}.\]
\end{prop}
\begin{proof}
In C.T.C's Wall list $K^{'}_{1,0}$ is the singularity defined by the ideal
\[\gen{ xy+z^{2},x^{2}+w^{2}+z^{2}y+\lambda y^{4}}\,\,\,\,\,\,\lambda \neq 0,\frac{1}{4}\]
with Milnor number $\mu=13$ and Tjurina number $ \tau=13.$
The versal deformation of $K^{'}_{1,0}$ is given by
\begin{dmath*}{}
\gen{xy+z^{2}+t_{1}w+t_{2}y+t_{3},x^{2}+w^{2}+z^{2}y+\lambda y^{4}+\lambda_{1}z^{2}+\lambda_{2}y^{2}z+\lambda_{3}yz+\lambda_{4}z+\lambda_{5}y^{5}+\lambda_{6}y^{4}+\lambda_{7}y^{3}+
\lambda_{8}y^{2}+\lambda_{9}y+\lambda_{10}}
\end{dmath*}
$K^{'}_{1,0}$ defines a weighted homogenous isolated complete intersection singularity with weights \[(w_{1},w_{2},w_{3},w_{4}) = (4,2,3,4)\]
 and the degrees \[(d_{1},d_{2})=(6,8).\]
The versal $\mu$-constant deformation of $K^{'}_{1,0}$ is given by
\[\gen{xy+z^{2},x^{2}+w^{2}+z^{2}y+\lambda y^{4}+\lambda_{5}y^{5}}.\]
Using the coordinate change $ x\rightarrow \xi^{4}x$, $y\rightarrow \xi^{2}y$, $z\rightarrow \xi^{3}z$,$ w\rightarrow \xi^{4}w$, we obtain
\[ \gen{xy+z^{2},x^{2}+w^{2}+z^{2}y+\lambda y^{4}+\lambda_{5}\xi^{2}y^{5}}.\]
Choosing $\xi$ such that $\lambda_{5}\xi^{2}= 1$ we obtain
\[\gen{xy+z^{2},x^{2}+w^{2}+z^{2}y+\lambda y^{4}+y^{5} }\]with Tjurina number $\tau = 12$.
\end{proof}
\begin{prop}
\label{K15}
The unimodular complete intersection surface singularities with Milnor number $\mu=15$ are $K^{'}_{15}$ with Tjurina number $\tau=15$ defined by the ideal \[\gen{xy+z^{2},x^{2}+w^{2}+zy^{3}}\] and $K^{'}_{15,1}$ with Tjurina number $\tau=14$ defined by the ideal \[\gen{xy+z^{2},x^{2}+w^{2}+zy^{2}+y^{6}}.\]
and $K^{'}_{15,2}$ with Tjurina number $\tau=13$ defined by the ideal
\[\gen{xy+z^{2},x^{2}+w^{2}+zy^{3}+y^{5}}.\]
\end{prop}
\begin{proof}
In C.T.C. Wall's list $K^{'}_{15}$ is the singularity defined by the ideal
\[\gen{xy+z^{2},x^{2}+w^{2}+zy^{3}}.\]
with Milnor number $\mu=15$ and Tjurina number $\tau=15$. The versal deformation of $K^{'}_{15}$ is given by
\begin{dmath*}{}
\gen{xy+z^{2}+t_{1}w+t_{2}y+t_{3},x^{2}+w^{2}+zy^{2}+\lambda_{1}yz^{2}+
\lambda_{2}z^{2}+\lambda_{3}yz+\lambda_{4}z+
\lambda_{5}y^{6}+\lambda_{6}y^{5}+\lambda_{7}y^{4}+
\lambda_{8}y^{3}+\lambda_{9}y^{2}+\lambda_{10}y+\lambda_{11}}.
\end{dmath*}
$K^{'}_{15}$ defines a weighted homogenous isolated complete intersection singularity with weights \[(w_{1},w_{2},w_{3},w_{4}) = (7,3,5,7)\]
and the degrees \[(d_{1},d_{2})=(10,14).\]
The versal $\mu$-constant deformation of $K^{'}_{15}$ is given by
\[\gen{ xy+z^{2},x^{2}+w^{2}+zy^{3}+\lambda_{5}y^{6}+\lambda_{6}y^{5}}.\]
If $\lambda_{6}\neq 0$
then
\[I= \gen{xy+z^{2},x^{2}+w^{2}+zy^{3}+uy^{5}},\,\, where\,\, u = \lambda_{5}y+\lambda_{6}.\]
Using coordinate change such that $x\rightarrow \xi^{7}x$, $y\rightarrow \xi^{3}y$, $z\rightarrow \xi^{5}z$
and $w\rightarrow \xi^{7}w$ we may assume
\[I= \gen{xy+z^{2},x^{2}+w^{2}+zy^{3}+\xi \bar{u}y^{5}}.\]
Choosing $\xi$ such that $\xi \bar{u} = 1$ we obtain
\[I =\gen{xy+z^{2},x^{2}+w^{2}+zy^{3}+y^{5}}\]
with Tjurina number $\tau=13$. If $\lambda_{6}= 0$  then again applying the same coordinate change we obtain
\[I = \langle xy+z^{2},x^{2}+w^{2}+zy^{3}+y^{6} \rangle\]
by choosing $\lambda_{5}\xi^{4}= 1$ with Tjurina number $\tau =14$.
\end{proof}
\begin{prop}
\label{K16}
The unimodular complete intersection surface singularities with Milnor number $\mu=16$   are $K^{'}_{16}$ with Tjurina number $\tau=16$ defined by the ideal
\[\gen{xy+z^{2},x^{2}+w^{2}+y^{5}}.\]
$K^{'}_{16,1}$ with Tjurina number $\tau=15$ defined by the ideal
\[\gen{xy+z^{2},x^{2}+w^{2}+y^{5}+z^{2}y^{3}}.\]
and $K^{'}_{16,2}$ with Tjurina number $\tau=14$ defined by the ideal
\[\gen{xy+z^{2},x^{2}+w^{2}+y^{5}+z^{2}y^{2}}.\]
\end{prop}
\newpage
\begin{proof}
The proof of Proposition \ref{K16} can be done similarly to the proof of Proposition \ref{K15}.
\end{proof}

Summarizing the results of the above propositions we complete the list of unimodular complete intersection
singularities in case of $\langle f,g \rangle$ having 2-jet with normal forms $\gen{xy+z^{2},w^{2}+x^{2}}$.
\begin{table}[H]
\[
\begin{array}{|l|c|c|c|}\hline
Type &Normal form & \mu  &\tau\\\hline
K_{10,1}^{'}   &   (xy+z^{2},w^{2}+x^{2}+y^{3}+yz^{2})    &  10    & 9 \\\hline
K_{11,1}^{'}   &   (xy+z^{2}, w^{2}+x^{2}+zy^{2}+y^{5})   & 11   &   10   \\\hline
K_{1,0,1}^{'}  &   (xy+z^{2}, w^{2}+x^{2}+yz^{2}+\lambda y^{4}+y^{5})& 13  &   12\\\hline
K_{15,1}^{'}   &   (xy+z^{2}, w^{2}+x^{2}+zy^{3}+y^{6})   & 15   &   14   \\\hline
K_{15,2}^{'}   &   (xy+z^{2}, w^{2}+x^{2}+zy^{3}+y^{5})   & 15   &   13   \\\hline
K_{16,1}^{'}   &   (xy+z^{2}, w^{2}+x^{2}+y^{5}+y^{3}z^{2})   &    16      &   15\\\hline
K_{16,2}^{'}   &   (xy+z^{2}, w^{2}+x^{2}+y^{5}+y^{2}z^{2}) & 16     &   14  \\\hline

\end{array}
\]
\caption{}
\label{Knew}
\end{table}
\begin{prop}
\label{Jcase}
Let $(V(\langle f,g\rangle ),0)\subseteq (\mathbb{C}^{4},0)$ be the germ of a complete intersection surface singularity. Assume it is not a hypersurface singularity and the $2$-jet of $\langle f,g\rangle$ has normal form $\langle xy+z^{2},w^{2}+x^{2}\rangle$. $(V(\langle f,g\rangle ),0)$ is unimodular if and only if it is isomorphic to a complete intersection in Tables \ref{Kold} and \ref{Knew}.
\end{prop}
\begin{proof}
The proof is a direct consequence of C.T.C. Wall's classification and Propositions \ref{K10} - \ref{K16}
\end{proof}
\begin{prop}
Let $ (\mathbf{V}(\gen{f,g}),0)\subseteq (\mathbb{C}^{4},0)$ be the germ of a complete intersection surface singularity. Assume it is not hypersurface singularity and the two jet of $\gen{f,g}$ has normal form $\gen{xy+z^{2},w^{2}+xz}$. $\mathbf{V}(\gen{f,g},0)$ is unimodular if and only if it is isomorphic to a complete intersection in table $3$.
\end{prop}

We summarize our approach in this case in Algorithm \ref{Ksing}.
\begin{algorithm}[H]
\caption{$\texttt{K}^{'}$\texttt{singularity(I)}}
\label{Ksing}
\begin{algorithmic}[1]
\REQUIRE $I=\gen{f,g}\subseteq \gen{x,y,z,w}^2\mathbb{C}[[x,y,z,w]]$ such that $2$-jet  of $I$\\             has normal form $(xy+z^{2},w^{2}+x^{2})$
\ENSURE the type of the singularity
\vspace{0.1cm}
\STATE compute $\mu=$Milnor number of $I$;
\STATE compute $\tau=$Tjurina number of $I$;
\STATE compute $B=$the singularity type of the strict transform of $I$ in the blowing up of $\gen{x,y,z,w}$
\IF{$\mu=10$ and $B=1$}
\IF{$\mu=\tau$}
\RETURN $(K^{'}_{10})$;
\ENDIF
\IF{$\mu - \tau = 1$}
\RETURN $(K^{'}_{10,1})$;
\ENDIF
\ENDIF
\IF{$\mu=11$ and $B=A[1]$}
\IF{$\mu=\tau$}
\RETURN $(K^{'}_{11})$;
\ENDIF
\IF{$\mu - \tau = 1$}
\RETURN $(K^{'}_{11,1})$;
\ENDIF
\ENDIF
\IF{$\mu=13$ and $B=A[3]$}
\IF{$\mu=\tau$}
\RETURN $(K^{'}_{1,0})$;
\ENDIF
\IF{$\mu - \tau = 1$}
\RETURN $(K^{'}_{1,0,1})$;
\ENDIF
\ENDIF
\IF{$\mu=15$ and $B=D[5]$}
\IF{$\mu=\tau$}
\RETURN $(K^{'}_{15})$;
\ENDIF
\IF{$\mu - \tau = 1$}
\RETURN $(K^{'}_{15,1})$;
\ENDIF
\IF{$\mu - \tau = 2$}
\RETURN $(K^{'}_{15,2})$;
\ENDIF
\ENDIF
\IF{$\mu=16$ and $B=E[6]$}
\IF{$\mu=\tau$}
\RETURN $(K^{'}_{16})$;
\ENDIF
\IF{$\mu - \tau = 1$}
\RETURN $(K^{'}_{16,1})$;
\ENDIF
\IF{$\mu - \tau = 2$}
\RETURN $(K^{'}_{16,2})$;
\ENDIF
\ENDIF
\IF{$\mu\neq \tau$}
\IF{$\mu-\tau=2$ and $\mu>13$}
\IF{$B=D[\mu-10]$}
\RETURN $(K^{'}_{1,\mu-13})$;
\ENDIF
\IF{$B=A[\mu-10]$}
\RETURN $(K^{b}_{1,\mu-13})$;
\ENDIF
\ENDIF
\ENDIF
\RETURN (not unimodular);
\end{algorithmic}
\end{algorithm}
\subsection{L singularities}
Assume the $2-jet$ of $\gen{f,g}$ has normal form $\gen{wz+xy,y^{2}+xz}$. According to C.T.C. Wall's classification the unimodular singularities are given in the table.

\begin{table}[H]
$$
\begin{array}{|c|c|c|c|}
  \hline
  Type & Normal\,\,\, form& \mu &\tau\\\hline
  L_{10} &\gen{wz+xy,y^{2}+xz+w^{3}} & 10 & 10 \\\hline
  L_{11} & \gen{wz+xy,y^{2}+xz+xw^{2}}  & 11&11 \\\hline
  L_{1,0} & \gen{wz+xy,y^{2}+xz+x^{2}w+\lambda w^{4}}  &11 & 11 \\\hline
  L_{15} & \gen{wz+xy,y^{2}+xz+xw^{3}}  & 15 & 16 \\\hline
  L_{16} &\gen{wz+xy,y^{2}+xz+w^{5}}  & 16 & 16 \\
  \hline
\end{array}
$$
\caption{}
\label{Lold}
\end{table}
\begin{prop}
\label{L10}
The unimodular complete intersection surface singularity with Milnor number $\mu=10$ and Tjurina number $\tau=10$ is $L_{10}$ defined by the ideal
\[\gen{wz+xy,y^{2}+xz+w^{3}}\]
and $L_{10,1}$ with Tjurina number $\tau=9$ is defined by the ideal
\[\gen{wz+xy,y^{2}+xz+w^{3}+yw^{2}}.\]
\end{prop}
\begin{proof}
In C.T.C. Wall's list the only unimodular complete intersection singularity with Milnor number $\mu=10$ is the singularity $L_{10}$ defined by the ideal $(wz+xy,y^{2}+xz+w^{3})$. The versal deformation of the above singularity is given by
\[\gen{xy+zw+Hw^{2}+Iw+J,y^{2}+xz+w^{3}+Ayw^{2}+Bw^{2}+Cyw+Dw+Ez+Fy+G}.\]
$L_{10}$ defines a weighted homogenous isolated complete intersection singularity with weights \[(w_{1},w_{2},w_{3},w_{4})=(5,6,7,4)\] and degrees
\[(d_{1},d_{2})=(11,12).\]
The versal $\mu-$constant deformation of $L_{10}$ is
\[\gen{wz+xy,y^{2}+xz+w^{3}+Ayw^{2}}.\]
Using the coordinate change
\[x\rightarrow\varepsilon^{5}x,\,\, y\rightarrow\varepsilon^{6}y,\,\,z\rightarrow\varepsilon^{7}z\,\, and\,\, w\rightarrow\varepsilon^{4}w\]
we obtain
\[\gen{wz+xy,y^{2}+xz+w^{3}+\varepsilon^{2}A^{'}yw^{2}}.\]
Take $\varepsilon$ such that $\varepsilon^{2}A^{'}=1$ then we obtain
\[ \gen{wz+xy,y^{2}+xz+w^{3}+yw^{2}}.\]
\end{proof}
\begin{prop}
\label{L11}
The unimodular complete intersection surface singularity with Milnor number $\mu=11$ and Tjurina number $\tau=11$ is $L_{11}$ defined by the ideal
\[\gen{wz+xy,y^{2}+xz+xw^{2}}\]
and $L_{11,1}$ with Tjurina number $\tau=10$ is defined by the ideal
\[\gen{wz+xy,y^{2}+xz+xw^{2}+w^{4}}.\]
\end{prop}
\begin{proof}
The proof of Proposition \ref{L11} can be done similarly to the proof of Proposition \ref{L10}.
\end{proof}
\begin{prop}
\label{L1,0}
The unimodular complete intersection surface singularity with Milnor number $\mu=13$ and with Tjurina number $\tau=13$ is $L_{1,0}$ defined by the ideal
\[\gen{wz+xy,y^{2}+xz+x^{2}w+\lambda y^{4}},\,\, \lambda\neq 0,-1 .\]
and $L_{1,0,1}$ with Tjurina number $\tau=12$ is defined by the ideal
\[\gen{wz+xy,y^{2}+xz+x^{2}w+\lambda w^{4}+w^{5}}.\]
\end{prop}
\begin{proof}
In C.T.C. Wall's list the only unimodular complete intersection singularity with Milnor number $13$ is the singularity $L_{1,0}$ defined by the ideal
\[\gen{wz+xy,y^{2}+xz+x^{2}w+\lambda w^{4}},\,\, \lambda\neq 0,-1 .\]
The versal deformation of the above singularity is given by
\begin{dmath*}{}
\gen{xy+zw+Jw^{3}+Kw^{2}+Lw+M,y^{2}+xz+x^{2}w+\lambda w^{4}+Aw^{5}+Bw^{4}+Cw^{3}+Dw^{2}+Eyw+Fw++Gz+Hy+I}.
\end{dmath*}
$L_{1,0}$ defines a weighted homogenous isolated complete intersection singularity with weights \[(w_{1},w_{2},w_{3},w_{4})=(3,4,5,2)\]
and degrees \[(d_{1},d_{2})=(7,8).\] The versal $\mu-$constant deformation of $L_{1,0}$ is \[\gen{wz+xy,y^{2}+xz+x^{2}w+\lambda w^{4}+Aw^{5}}.\]
Now using the coordinate change $x\rightarrow\varepsilon^{3}x$, $y\rightarrow\varepsilon^{4}y$, $z\rightarrow\varepsilon^{5}z$ and $w\rightarrow\varepsilon^{2}w$. We obtain \[\gen{wz+xy,y^{2}+xz+x^{2}w+\lambda w^{4}+\varepsilon^{2}A^{'}w^{5}}.\]
 Choosing $\varepsilon$ such that $\varepsilon^{2}A^{'}=1$. We obtain \[\gen{wz+xy,y^{2}+xz+x^{2}w+\lambda w^{4}+w^{5}}\]
 with Tjurina number $\tau=9$.
 \end{proof}
\begin{prop}
\label{L15}
The unimodular complete intersection surface singularity with Milnor number $\mu=15$ and with Tjurina number $\tau=15$ is $L_{15}$ defined by the ideal
\[\gen{wz+xy,y^{2}+xz+xw^{3}},\]
$L_{15,1}$ with Tjurina number $\tau=14$ is defined by the ideal
\[\gen{wz+xy,y^{2}+xz+xw^{3}+w^{6}}\]
and $L_{15,2}$ with Tjurina number $\tau=13$ is defined by the ideal
\[\gen{wz+xy,y^{2}+xz+xw^{3}+w^{5}}.\]
\end{prop}
\begin{proof}
In C.T.C. Wall's list the only unimodular complete intersection singularity with Milnor number $\mu=15$ is the singularity $L_{15}$ defined by the ideal
\[\gen{wz+xy,y^{2}+xz+xw^{3}}.\]
The versal deformation of the above singularity is given by
\begin{dmath*}{}
\gen{xy+zw+Lw^{3}+Mw^{2}+Nw+O,y^{2}+xz+xw^{3}+Aw^{6}+Bw^{5}
+Cw^{4}+Dw^{3}+Ew^{2}y+Fw^{2}+Gyw+Hw+
Iz+Jy+K}.
\end{dmath*}
$L_{15}$ defines a weighted homogenous isolated complete intersection singularity with weights \[(w_{1},w_{2},w_{3},w_{4})=(5,7,9,3)\]
 and degrees
 \[(d_{1},d_{2})=(12,14).\]
 The versal $\mu-$constant deformation of $L_{15}$ is
 \[\gen{wz+xy,y^{2}+xz+xw^{3}+Aw^{6}+Bw^{5}}.\]
 Let $B\neq0$,
\[\gen{xy+wz,y^{2}+xz+xw^{3}+w^{5}(Aw+B)}.\]
Let $\mu_{B}=Aw+B$, Now using the coordinate change $x\rightarrow\varepsilon^{5}x$, $y\rightarrow\varepsilon^{7}y$, $z\rightarrow\varepsilon^{9}z$ and $w\rightarrow\varepsilon^{3}w$, we obtain \[\gen{wz+xy,y^{2}+xz+xw^{3}+\varepsilon\acute{\mu_{B}}w^{5}}.\]
Choosing $\varepsilon$ such that $\varepsilon\acute{\mu_{B}}=1$. So we obtain
\[\gen{wz+xy,y^{2}+xz+xw^{3}+w^{5}}\] with Tjurina number $\tau=13$.
Now we if $B=0$ then it becomes
\[\gen{wz+xy,y^{2}+xz+xw^{3}+Aw^{6}}.\]
Now using again the same coordinate change we obtain
 \[\gen{wz+xy,y^{2}+xz+xw^{3}+\varepsilon^{4}A^{'}w^{6}}.\]
Choosing $\varepsilon$ such that $\varepsilon^{4}A^{'}=1$ so we obtain
\[\gen{wz+xy,y^{2}+xw^{3}+w^{6}}\] with Tjurina number $\tau=14$.
\end{proof}
\begin{prop}
\label{L16}
The unimodular complete intersection surface singularity with Milnor number $\mu=16$ and Tjurina number $\tau=16$ is $L_{16}$ defined by the ideal
\[\gen{wz+xy,y^{2}+xz+w^{5}},\]
 $L_{16,1}$ with Tjurina number $\tau=15$ is defined by the ideal
\[ \gen{wz+xy,y^{2}+xz+w^{5}+yw^{4}}\] and and $L_{16,2}$ with Tjurina number $\tau=14$ is defined by the ideal \[ \gen{wz+xy,y^{2}+xz+w^{5}+yw^{3}}.\]
\end{prop}
Summarizing the above results we have the following table.
\begin{table}[H]
$$
\begin{array}{|c|c|c|c|}\hline

  Name &Normal form  & \mu &\tau \\\hline
  L_{10,1} &\gen{wz+xy,y^{2}+xz+w^{3}+yw^{2}} &10 &9 \\\hline
  L_{11,1} &\gen{xy+zw,y^{2}+xz+xw^{2}+w^{4}} &11 &10 \\\hline
  L_{1,0,1} &\gen{xy+zw,y^{2}+xz+x^{2}w+w^{4}+w^{5}} &13 &12 \\\hline
  L_{15,1} &\gen{xy+zw,y^{2}+xz+xw^{3}+w^{6}} & 15 &14 \\\hline
  L_{15,2}&\gen{xy+zw,y^{2}+xz+xw^{3}+w^{5}} &15 & 13 \\\hline
  L_{16,1} &\gen{xy+zw,y^{2}+xz+w^{5}+yw^{4}} &16 &15 \\\hline
  L_{16,2} & \gen{xy+zw,y^{2}+xz+w^{5}+yw^{3}} &16 &14 \\\hline
\end{array}
$$
\caption{}
\label{Lnew}
\end{table}
\begin{prop}
\label{Lcase}
Let $(V(\langle f,g\rangle ),0)\subseteq (\mathbb{C}^{4},0)$ be the germ of a complete intersection surface singularity. Assume it is not a hypersurface singularity and the $2$-jet of $\langle f,g\rangle$ has normal form $\langle wz+xy,y^{2}+xz\rangle$. $(V(\langle f,g\rangle ),0)$ is unimodular if and only if it is isomorphic to a complete intersection in Tables \ref{Lold} and \ref{Lnew}.
\end{prop}
\begin{algorithm}[H]
\caption{\texttt{Lsingularity(I)}}
\label{Lsing}
\begin{algorithmic}[1]
\REQUIRE $I=\gen{f,g}\in \gen{x,y,z,w}^2\mathbb{C}[[x,y,z,w]]$ such that $2$-jet  of $I$\\             has normal form $\gen{wz+xy,y^{2}+xz+w^{3}}$
\ENSURE the type of the singularity
\vspace{0.1cm}
\STATE compute $\mu=$Milnor number of $I$, $\tau=$Tjurina number of $I$ and $B=$the singularity type of the strict transform of $I$ in the blowing up of $\gen{x,y,z,w}$
\IF{$\mu=10$ and $B=1$}
\IF{$\mu=\tau$}
\RETURN $(L_{10})$;
\ENDIF
\IF{$\mu - \tau = 1$}
\RETURN $(L_{10,1})$;
\ENDIF
\ENDIF
\IF{$\mu=11$ and $B=A[1]$}
\IF{$\mu=\tau$}
\RETURN $(L_{11})$;
\ENDIF
\IF{$\mu - \tau = 1$}
\RETURN $(L_{11,1})$;
\ENDIF
\ENDIF
\IF{$\mu=13$ and $B=A[3]$}
\IF{$\mu=\tau$}
\RETURN $(L_{1,0})$;
\ENDIF
\IF{$\mu - \tau = 1$}
\RETURN $(L_{1,0,1})$;
\ENDIF
\ENDIF
\IF{$\mu=15$ and $B=D[5]$}
\IF{$\mu=\tau$}
\RETURN $(L_{15})$;
\ENDIF
\IF{$\mu - \tau = 1$}
\RETURN $(L_{15,1})$;
\ENDIF
\IF{$\mu - \tau = 2$}
\RETURN $(L_{15,2})$;
\ENDIF
\ENDIF
\IF{$\mu=16$ and $B=E[6]$}
\IF{$\mu=\tau$}
\RETURN $(L_{16})$;
\ENDIF
\IF{$\mu - \tau = 1$}
\RETURN $(L_{16,1})$;
\ENDIF
\IF{$\mu - \tau = 2$}
\RETURN $(L_{16,2})$;
\ENDIF
\ENDIF
\IF{$\mu-\tau=2$ and $\mu>13$}
\IF{$B=D[\mu-10]$}
\RETURN $(L_{1,\mu-13})$;
\ENDIF
\IF{$B=A[\mu-10]$}
\RETURN $(L^{b}_{1,\mu-13})$;
\ENDIF
\ENDIF
\RETURN ($not\,\,unimodular$);
\end{algorithmic}
\end{algorithm}
\subsection{M singularities}
Assume the $2-jet$ of $\gen{f,g}$ has normal form $\gen{wy+x^{2}-z^{2},wx}$. In this case According to C.T.C.Wall's classification the unimodular surface singularities with Milnor number $\mu$ are given in the table below.
\begin{table}[H]
\[
\begin{array}{|l|c|c|c|}\hline
 Name & Normal form&\mu & \tau \\\hline
 M_{11}   &   \gen{wy+x^{2}-z^{2},wx+y^{3}} & 11& 11 \\\hline
 M_{1,0}  &   \gen{wy+x^{2}-z^{2},wx+y^{2}x+\lambda y^{2}z} &13 & 13\\\hline
 M_{1,i}  &   \gen{wy+x^{2}-z^{2},wx+y^{2}x+ zI_{i+1}(z,y)} &13+i &13+i-2\\\hline
 M_{15}   &   \gen{wy+x^{2}-z^{2},wx+y^{4}} &15 & 15\\\hline
\end{array}
\]
\caption{}
\label{Mold}
\end{table}

\begin{prop}
\label{M11}
The unimodular complete intersection surface singularities with Milnor number $\mu=11$ are $M_{11}$ with Tjurina number $\tau=11$ defined by the ideal
\[\gen{wy+x^{2}-z^{2},wx+y^{3}}\]
and $M_{11,1}$ with Tjurina number $\tau=10$ defined by the ideal
\[\gen{wy+x^{2}-z^{2},wx+y^{3}+ y^{2}w}.\]
\end{prop}
\begin{proof}
In C.T.C. Wall's list $M_{11}$ is the singularity defined by the ideal
\[\gen{ wy+x^{2}-z^{2},wx+y^{3}}\]
with Milnor number $\mu=11$ and Tjurina number $\tau=11$. The versal deformation of $M_{11}$ is
\[\gen{wy+x^{2}-z^{2}+t_{1}y^{2}+t_{2}y+t_{3},wx+y^{3}+\lambda_{1}y^{2}w+\lambda_{2}yw+\lambda_{3}w+\lambda_{4}yz+\lambda_{5}z
+\lambda_{6}y^{2}+\lambda_{7}y+\lambda_{8}}.\]
$M_{11}$ defines a weighted homogenous isolated complete intersection singularity with weights
 \[(w_{1},w_{2},w_{3},w_{4}) = (4,3,4,5)\]
and the degrees
\[(d_{1},d_{2})=(8,9).\]
The versal $\mu$-constant deformation of $M_{11}$ is given by
\[\gen{wy+x^{2}-z^{2},wx+y^{3}+\lambda_{1}y^{2}w}.\]
Using the coordinate change $x\rightarrow \xi^{4}x$, $y\rightarrow \xi^{3}y$, $z\rightarrow \xi^{4}z$
and $w\rightarrow \xi^{5}w$ we obtain
\[\gen{wy+x^{2}-z^{2},wx+y^{3}+\lambda_{1}\xi y^{2}w}.\]
 Choosing $\xi$ such that $\lambda_{1}\xi= 1$ we obtain
\[\gen{wy+x^{2}-z^{2},wx+y^{3}+ y^{2}w}\]
with different Tjurina number $\tau = 10$ from $M_{11}$.
\end{proof}
\begin{prop}
\label{M13}
The unimodular complete intersection surface singularities with Milnor number $\mu=13$   are $M_{1,0}$ with Tjurina number $\tau=13$ defined by the ideal
 \[\gen{wy+x^{2}-z^{2},wx+y^{2}x+\lambda y^{2}z}\]
 $M_{1,0,1}$ with Tjurina number $\tau=12$ defined by the ideal
\[\gen{wy+x^{2}-z^{2},wx+y^{2}x+\lambda y^{2}z+y^{3}z}.\]
\end{prop}
\begin{proof}
In the list of C.T.C Wall $M_{1,0}$ is the singularity defined by the ideal
 \[\gen{wy+x^{2}-z^{2},wx+y^{2}x+\lambda y^{2}z}.\]
with Milnor number $\mu=13$ and Tjurina number $ \tau=13$. The versal deformation of $M_{1,0}$ is
\begin{dmath*}{}
\gen{wy+x^{2}-z^{2}+t_{1}y^{2}+t_{2}y+t_{3},wx+y^{2}x+\lambda y^{2}z+\lambda_{1}yw+\lambda_{2}w+\lambda_{3}y^{3}z+\lambda_{4}y^{2}z+\lambda_{5}yz+\lambda_{6}z+\lambda_{7}y^{3}+
\lambda_{8}y^{2}+\lambda_{9}y+\lambda_{10}}.
\end{dmath*}
 $M_{1,0}$ defines a weighted homogenous isolated complete intersection singularity with weights  \[(w_{1},w_{2},w_{3},w_{4}) = (6,4,6,8)\]
and the degrees
\[(d_{1},d_{2})=(12,14)\]
The versal $\mu$-constant deformation of $M_{1,0}$ is given by
 \[\gen{wy+x^{2}-z^{2},wx+y^{2}x+\lambda y^{2}z+\lambda_{3}y^{3}z}.\]
 Using the coordinate change $x\rightarrow \xi^{6}x$, $y\rightarrow \xi^{4}y$,
$z\rightarrow \xi^{6}z$, $w\rightarrow \xi^{8}w$ we obtain
\[\gen{wy+x^{2}-z^{2},wx+y^{2}x+\lambda y^{2}z+\lambda_{3}\xi^{4}y^{3}z}.\]
 Choosing $\xi$ such that $\lambda_{3}\xi^{4}= 1$
\[\gen{wy+x^{2}-z^{2},wx+y^{2}x+\lambda y^{2}z+y^{3}z}.\]
with Tjurina number $\tau = 12$.
\end{proof}
\begin{prop}
\label{M15}
The unimodular complete intersection surface singularities with Milnor number $\mu=15$   are $M_{15}$ with Tjurina number $\tau=15$ defined by the ideal
\[\gen{wy+x^{2}-z^{2},wx+y^{4}},\]
$M_{15,1}$ with Tjurina number $\tau=14$ defined by the ideal
\[\gen{wy+x^{2}-z^{2},wx+y^{4}+y^{2}w}\]
and $M_{15,2}$ with Tjurina number $\tau=13$ defined by the ideal
\[\gen{wy+x^{2}-z^{2},wx+y^{4}+y^{3}w }.\]
\end{prop}
\begin{proof}
In the list of C.T.C Wall $M_{15}$ is the singularity defined by the ideal
\[\gen{ wy+x^{2}-z^{2},wx+y^{4}}\]
with Milnor number $\mu=15$ and Tjurina number $ \tau=15.$
The versal deformation of $M_{15}$ is
\begin{dmath*}{}
\gen{wy+x^{2}-z^{2}+t_{1}y^{2}+t_{2}y+t_{3},wx+y^{4}+\lambda_{1}y^{3}w+\lambda_{2}y^{2}w+\lambda_{3}yw+\lambda_{4}w
+\lambda_{5}y^{2}z+\lambda_{6}yz+\lambda_{7}z+\lambda_{8}y^{3}+\lambda_{9}y^{2}+\lambda_{10}y+\lambda_{11}}.
\end{dmath*}
 $M_{15}$ defines a weighted homogenous isolated complete intersection singularity with weights
\[(w_{1},w_{2},w_{3},w_{4}) = (5,3,5,7)\]
and the degrees
\[(d_{1},d_{2})=(10,12).\]
\newpage
The versal $\mu$-constant deformation of $M_{15}$ is given by
\[\gen{ wy+x^{2}-z^{2},wx+y^{4}+\lambda_{1}y^{3}w+\lambda_{2}y^{2}w}.\]
If $\lambda_{2}\neq 0$ then we have
\[I= \gen{wy+x^{2}-z^{2},wx+y^{4}+uy^{2}w}\]
where $u = \lambda_{1}y+\lambda_{2}$.
Using the coordinate change $ x\rightarrow \xi^{6}x$, $y\rightarrow \xi^{4}y$, $ z\rightarrow \xi^{6}z$ and $ w\rightarrow \xi^{8}w$ we obtain
\[\gen{wy+x^{2}-z^{2},wx+y^{4}+u\xi^{}y^{2}w}.\]
Choosing $ u\xi^{2} = 1$ we obtain
\[\gen{wy+x^{2}-z^{2},wx+y^{4}+y^{2}w}\]
with Tjurina number $\tau = 14$. If $\lambda_{2}= 0$ then again by the same transformation we obtain
\[ \gen{wy+x^{2}-z^{2},wx+y^{4}+y^{3}w}\]
with Tjurina number $\tau = 13$ by choosing
$\lambda_{1}\xi^{2}= 1$.
\end{proof}
Summarizing the results of the above prepositions we complete the list of unimodular complete intersection singularities in case of $\langle f,g \rangle$ having 2-jet with normal forms  $\gen{wy+x^{2}-z^{2},wx}$.
\bigskip
\begin{table}[H]
\[
\begin{array}{|l|c|c|c|}
\hline Type &Normal form & \mu  &\tau\\
\hline M_{11,1}       &   \gen{wy+x^{2}-z^{2},wx+y^{3}+y^{2}w}& 11     &   10  \\
\hline M_{1,0,1}      &   \gen{wy+x^{2}-z^{2},wx+y^{2}x+\lambda y^{2}z+y^{3}z}& 13     &   12  \\
\hline M_{15,1}       &   \gen{wy+x^{2}-z^{2},wx+y^{4}+y^{3}w}& 15     &   14  \\
\hline M_{15,2}       &   \gen{wy+x^{2}-z^{2},wx+y^{4}+y^{2}w}& 15     &   13  \\
\hline
\end{array}
\]
\bigskip
\caption{}
\label{Mnew}
\end{table}
\begin{prop}
\label{Mcase}
Let $(V(\langle f,g\rangle ),0)\subseteq (\mathbb{C}^{4},0)$ be the germ of a complete intersection surface singularity. Assume it is not a hypersurface singularity and the $2$-jet of $\langle f,g\rangle$ has normal form $\langle wy+x^{2}-z^{2},wx\rangle$. $(V(\langle f,g\rangle ),0)$ is unimodular if and only if it is isomorphic to a complete intersection in Tables \ref{Mold} and \ref{Mnew}.
\end{prop}
\begin{proof}
The proof is a direct consequence of C.T.C. Wall's classification and Propositions \ref{M11} - \ref{M15}
\end{proof}
\newpage
We summarize our approach in this case in Algorithm \ref{Msing}.
\begin{algorithm}[H]
\caption{\texttt{Msingularity(I)}}
\label{Msing}
\begin{algorithmic}[1]
\REQUIRE $I=\gen{f,g}\subseteq \gen{x,y,z,w}^2\mathbb{C}[[x,y,z,w]]$ such that $2$-jet  of $I$\\             has normal form $\gen{2wy+x^{2}-z^{2}, 2wx}$
\ENSURE the type of the singularity
\vspace{0.1cm}
\STATE compute $\mu=$Milnor number of $I$;
\STATE compute $\tau=$Tjurina number of $I$;
\STATE compute $B=$the singularity type of the strict transform of $I$ in the blowing up of $\gen{x,y,z,w}$
\IF{$\mu=11$ and $B=1$}
\IF{$\mu=\tau$}
\RETURN $(M_{11})$;
\ENDIF
\IF{$\mu - \tau = 1$}
\RETURN $(M_{11,1})$;
\ENDIF
\ENDIF
\IF{$\mu=13$ and $B=A[1]$}
\IF{$\mu=\tau$}
\RETURN $(M_{1,0})$;
\ENDIF
\IF{$\mu - \tau = 1$}
\RETURN $(M_{1,0,1})$;
\ENDIF
\ENDIF
\IF{$\mu=15$ and $B=D[5]$}
\IF{$\mu=\tau$}
\RETURN $(M_{15})$;
\ENDIF
\IF{$\mu - \tau = 1$}
\RETURN $(M_{15,1})$;
\ENDIF
\IF{$\mu - \tau = 2$}
\RETURN $(M_{15,2})$;
\ENDIF
\IF{$\mu\neq \tau$}
\IF{$\mu-\tau=2$ and $\mu>13$}
\IF{$B=A[\mu-11]$}
\RETURN $(M_{1,\mu-13})$;
\ENDIF
\ENDIF
\ENDIF
\ENDIF
\RETURN (not unimodular);
\end{algorithmic}
\end{algorithm}
\section{Main Algorithm}
\label{main}
Now we give our main Algorithm \ref{Surface} by using the all Algorithms given in section \ref{ctcwall} by which we can compute the type of the singularity.
\begin{algorithm}[H]
\caption{\texttt{classifyicis2(I)}[Unimodular surface singularities]} \label{Surface}
\begin{algorithmic}[1]
\REQUIRE $I=\langle f,g\rangle  \subseteq \langle x,y,z,w\rangle^2\mathbb{C}[[x,y,z]]$ isolated complete intersection curve singularity.
\ENSURE The type of the singularity $(V(I),0)$.
\vspace{0.1cm}
\STATE compute $I_{2}$ the $2$-jet of $I$;
\STATE compute $I_{2}=\bigcap^s_{i=1}Q_{i}$  the irredundant primary decomposition over $\mathbb{C}$;
\STATE compute $d_{i}=$Krull dimension of $\mathbb{C}[x,y,z,w]/Q_{i}$;
\STATE compute $h_{i}\in \mathbb{Q}[t]$ the Hilbert polynomial corresponding to each $Q_{i}$;
\STATE compute $t$ number of absolute prime ideals of $I_{2}$ and $j_{i}$ the number of conjugates.
\IF {$s=4$}
\IF {$d_{1}=d_{2}=d_{3}=d_{4}=1$}
\IF{$h_{1}=h_{2}=h_{3}=h_{4}=1+t$}
\RETURN $\texttt{Isingularity(I)}$; via Algorithm \ref{Ising}
\ENDIF
\ENDIF
\ENDIF
\IF {$s=3$}
\IF {$d_{1}=d_{2}=d_{3}=2$}
\IF{$h_{1}=1+2t$ and $h_{2}=h_{3}=1+t$}
\IF{$t=3$, $j_{1}=j_{2}=j_{3}=1$}
\IF{$ Q_{3}\subseteq Q_{1}+ Q_{2}\,\,and\,\, Q_{3}\,\,has\,\,a\,\,generator\,\,of\,\,order\,\,2$}
\RETURN $\texttt{Msingularity(I)}$; via Algorithm \ref{Msing}
\ENDIF
\IF{$Q_{3}\nsubseteq Q_{1}+ Q_{2} \,\,and\,\, Q_{3}\,\,has\,\,a\,\,generator\,\,of\,\,order\,\,2$}
\RETURN $\texttt{Tsingulrity(I)}$; via Algorithm \ref{Tsing}
\ENDIF
\ENDIF
\ENDIF
\ENDIF
\ENDIF
\IF {$s=2$}
\IF {$d_{1}=d_{2}=2$}
\IF{$h_{1}=1+t=h_{2}=1+3t$}
\RETURN $\texttt{Lsingularity(I)}$; via Algorithm \ref{Lsing}
\ENDIF
\IF{$h_{1}=h_{2}=1+2t$}
\RETURN $\texttt{Tsingularity(I)}$; via Algorithm \ref{Tsing}
\ENDIF
\IF{$h_{1}=h_{2}=1+2t$}
\RETURN $\texttt{Tsingularity(I)}$; via Algorithm \ref{Tsing}
\ENDIF
\ENDIF
\ENDIF
\IF {$s=1$}
\IF {$d_{1}=2$ and $h_{1}=4t$}
\IF{$t=1$, $j_{1}=1$}
\IF{$A_{2}$ after blowing up}
\RETURN $\texttt{J}'\texttt{singularity(I)}$; via Algorithm \ref{Jsing}
\ENDIF
\IF{$A_{1}$ after blowing up}
\RETURN $\texttt{Tsingularity(I)}$; via Algorithm \ref{Tsing}
\ENDIF
\ENDIF
\IF{$t=2$ and $j_{1}=1,j_{2}=1$}
\RETURN $\texttt{Tsingularity(I)}$; via Algorithm \ref{Tsing}
\ENDIF
\IF{$t=1$ and $j_{1}=2$}
\RETURN $\texttt{K}'\texttt{singularity(I)}$; via Algorithm \ref{Ksing}
\ENDIF
\ENDIF
\ENDIF
\RETURN $(not\quad unimodular)$;
\end{algorithmic}
\end{algorithm}

{\bf Acknowledgements}
The part of this work is carried out at Kaiserslautern University, Germany. We are thankful to DAAD Germany for the financial support.

\end{document}